\documentclass[reqno]{amsart}
\newcommand \datum {Version of October 20, 2024}
\usepackage{amsfonts,amssymb,amsmath}
\usepackage{amsthm,amscd}
\usepackage{latexsym}
\usepackage{graphicx}

\usepackage{hyperref}
\usepackage{url}
\usepackage{graphicx}
\usepackage[dvipsnames]{xcolor}
\usepackage{enumerate}
\numberwithin{equation}{section}
\theoremstyle{plain}
 \newtheorem{theorem}{Theorem}[section]
  
 \newtheorem{lemma}[theorem]{Lemma}
\newtheorem{claim}[theorem]{Claim}
\theoremstyle{definition}
 \newtheorem{definition}[theorem]{Definition}
 \newtheorem{remark}[theorem]{Remark}
\theoremstyle{remark}

\newcommand \figwidthcoeff{0.95}
\newcommand \asys{\mathbf A}
\newcommand \Equ[1] {\textup{Equ}(#1)}
\newcommand \Quo[1] {\textup{Quo}(#1)}
\newcommand \Part[1] {\textup{Part}(#1)}
\newcommand \eq {\textup{eq}}
\newcommand \geneq {\overline{\textup{eq}}}
\newcommand \qu {\textup{qu}}
\newcommand \at {\textup{at}}
\newcommand \peq {\geneq\kern 1pt '}
\newcommand \pat {\textup{at}'}
\newcommand \pwedge {\wedge'}
\newcommand \pvee {\vee'}
\newcommand\textquotat[1]{`{`{#1}'}'} 
\newcommand \latop[1] {\mathbf 1_{#1}}
\newcommand \labot[1] {\mathbf 0_{#1}}
\newcommand \inv[1]  {#1^{-1}}
\newcommand \ialpha{\inv\alpha}
\newcommand \ibeta{\inv\beta}
\newcommand \igamma{\inv\gamma}
\newcommand \idelta{\inv\delta}
\newcommand\num[1]{\textup{blnum}(#1)}
\newcommand \deqref[2]{\overset{(\ref{#1},\ref{#2})}=}
\newcommand \shdeqref[2]{\kern-1pt \underset {\scriptscriptstyle{(\ref{#2})}} {\overset{ \scriptscriptstyle {\ref{#1})}  }=}\kern -1pt}

\renewcommand \phi{\varphi}
\newcommand \Nnul {{\mathbb N_0}}
\newcommand \Nplu {{\mathbb N^+}}
\newcommand{\tbf}{\textbf}
\newcommand{\set}[1]{\{#1\}}

\newcommand \red[1]{{\textcolor{red}{#1}\color{black}}}
\newcommand \blue[1]{{\textcolor{blue}{#1}}}
 
\begin{document}

\title[Four generators of an equivalence lattice]{Four generators of an equivalence lattice with consecutive block counts}

\author{G\'abor Cz\'edli}
\address{University of Szeged, Bolyai Institute. 
Szeged, Aradi v\'ertan\'uk tere 1, HUNGARY 6720, 
\href{http://www.math.u-szeged.hu/~czedli/}{http://www.math.u-szeged.hu/\textasciitilde{}czedli/} }
\email{czedli@math.u-szeged.hu}
\thanks{This research was supported by the National Research, Development and Innovation Fund of Hungary, under funding scheme K 138892. \hfill\red{\datum}}

\begin{abstract}
The \emph{block count} of an equivalence $\mu\in\Equ A$ is the number $\num\mu$ of blocks of (the partition corresponding to) $\mu$. We say that $X=\set{\mu_1,\mu_2,\mu_3,\mu_4}$ is a 
\emph{four-element generating set of $\Equ A$ with consecutive block counts} if
$X$ generates $\Equ A$ and $\num{\mu_{1+i}}=\num{\mu_1} +i$ for $i\in\set{1,2,3}$. We prove that if the number of elements of a finite set $A$ is six or at least eight, then  $\Equ A$ has a four-element generating set with consecutive block counts.
Also,  we present a historical remark on the connection between equivalence lattices and quasiorder lattices.
\end{abstract}

\dedicatory{Dedicated to my esteemed coauthors, Honorary Professors L\'aszl\'o Szab\'o on his seventy-fifth birthday and Lajos Klukovits on his eightieth birthday.}

\maketitle

\subsection*{Preface} This paper is probably self-contained for those who know the concept of a lattice as an algebraic structure. 
Our goal is two-fold. First, we present a historical remark on the connection between equivalence lattices and quasiorder lattices; see after the proof of Claim \ref{claim:mc95}. 
Second,  we prove a new theorem, Theorem \ref{thm:main}, which corresponds to the title of the paper.

\section{Introduction and a historical remark}\label{sect:intro}

We begin with some notations and well-known definitions. The set of \emph{equivalences} (in other words, \emph{equivalence relations}, that is, reflexive, symmetric, and transitive relations) of a set $A$ will be denoted by $\Equ A$. 
With intersections and the transitive hulls of unions acting as meets and joins, respectively, $\Equ A$ is a \emph{lattice}, the \emph{equivalence lattice} of (or over) $A$; the notation $\Equ A$ will stand for this lattice, too. By the canonical bijective correspondence between equivalences and partitions of a set, $\Equ A$ is isomorphic to the \emph{partition lattice} $\Part A$ of $A$, which consists of all partitions of $A$.  We will often consider equivalences as partitions. 
For $X\subseteq Y$, we say that $X$ is a \emph{proper} subset of $Y$ if $X\neq Y$. A \emph{sublattice} or a \emph{complete sublattice} of $\Equ A$ is a nonempty subset that is closed with respect to binary joins and meets or to arbitrary joins and meets, respectively. 
A subset $X$ of $\Equ A$ is a \emph{generating set} or a \emph{complete-generating set} of $\Equ A$ if there is no proper sublattice $Y$ or a proper complete sublattice $Y$ of $\Equ A$, respectively, such that $X\subseteq Y$. 
\emph{Quasiorders} are reflexive and symmetric relations. The quasiorders of a set $A$ form a lattice, the \emph{quasiorder lattice} $\Quo A$ of $A$. Note that $\Equ A$ is a complete sublattice of $\Quo A$. 

In the middle of the seventies, Henrik Strietz proved that for any finite set $A$ with $|A|\geq 3$, $\Equ A$ is \emph{four-generated},  that is, it has a four-element generating set; see Strietz \cite{strietz75}--\cite{strietz77}.  
Since Strietz's work, more than a dozen papers have been devoted to four-element  (or small) generating sets of equivalence lattices and quasiorder lattices; for details, see the \textquotat{References} section here and the bibliographic sections and the survey parts of the papers listed there. Hence, instead of giving another survey, we focus only on the connection between the small generating sets of $\Equ A$ and those of $\Quo A$. In one direction, 
we recall  an important statement from \cite[page 61]{kulin}; see also Lemma 2.1 of \cite{czedlikulin}, where the original lemma is recalled. 

\begin{lemma}[Kulin's Lemma]\label{lemma:kulin}
If $A$ is an arbitrary set with at least three elements and $S$ is a complete sublattice of $\Quo A$ such that $\Equ A$ is a proper subset of $S$, then $S=\Quo A$.
\end{lemma}

\begin{figure}[ht] 
\centerline{ \includegraphics[width=\figwidthcoeff\textwidth]{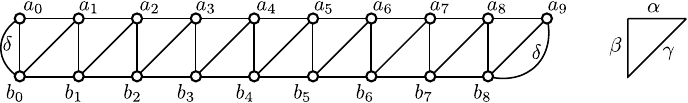}} \caption{Z\'adori's construction for $|A|=19$}\label{figa}
\end{figure}

In other directions, neither any connection nor the forthcoming Claim \ref{claim:mc95} has been published before. To present such a connection of historical value, let $|A|=19$; the case of $|A|=2k+1 \geq 5$ would be similar. The construction visualized by Figure \ref{figa} is taken from  Z\'adori \cite{zadori}.

\begin{claim}[\cite{zadori}; exemplifying the odd case of Z\'adori's construction]\label{claim:Z19}
If $|A|=19$, then $\Equ A$ has a four-element generating set.
\end{claim}

For later reference, we present Z\'adori's proof and his generating set.

\begin{proof}
For $p,q\in A$, the smallest equivalence collapsing $p$ and $q$ is an atom in $\Equ A$; we denote it by $\at(p,q)$. So $(x,y)\in\at(p,q)$ if and only if $x=y$ or $\set{x,y}=\set{p,q}$. 
Denote the elements of $A$ as follows:
$A=\{a_0,a_1,\dots,a_9$, $b_0,b_1,\dots,b_8\}$; see Figure \ref{figa}. 
The figure defines a subset $X:=\{\alpha,\beta,\gamma,\delta\}$ of the equivalence lattice $\Equ A$ as follows. Assume that the horizontal edges, the vertical edges, and the slanted straight edges of the graph are labeled by $\alpha$, $\beta$, and $\gamma$, respectively. To avoid a crowded figure, these labels are not indicated in the figure, but the triangle on the right reminds us of this convention. There are also two $\delta$-labeled edges, which are drawn as curves. For $\epsilon\in X$,  the figure defines $\epsilon$ as follows; walks of length zero are allowed.
\begin{equation}
\epsilon:=\{(x,y)\in A^2: \text{we can walk from }x\text{ to }y\text{ along }\epsilon\text{-colored edges}\}.
\label{eq:fgdFnwQvl}
\end{equation}
For example, $\set{b_1,a_2}$ is a block of $\gamma$ and $\set{b_0,\dots,b_8}$ is a block of $\alpha$. Let $S$ be the sublattice generated by $X$. In Figure \ref{figb}, where $A$ is drawn three times, some equivalences are given by their non-singleton blocks. The meanings of these blocks, with different geometric orientations, line styles, and colors, are defined on the right of the figure. For example, $\rho_0=\at(a_0,b_0)$ and $\lambda'_1=\at(a_9,a_8)\vee \at(a_8,a_7)\vee \at(b_8,b_7)$. 
We can easily show that, in this order, 
$\rho_0$, $\rho'_0$,  $\rho''_0$, $\rho_1$,   $\rho'_1$, $\rho''_1$, 
 $\rho_2$,   $\rho'_2$, $\rho''_2$,  $\rho_3$,   $\rho'_3$, $\rho''_3$, $\rho_4$, \dots 
belong to $S$, since each of them is expressible from the generators and the earlier ones. 
Indeed, $\rho_0=\beta\wedge\delta$ and, for $i=0,1,2,\dots$,  we have that 
$\rho'_i=(\rho_i\vee \gamma)\wedge \alpha$, $\rho''_i=(\rho'_i\vee \beta)\wedge \gamma$, and
 $\rho_{i+1}=\bigl(((\rho''_i \vee\beta)\wedge\alpha)\vee \rho''_i\bigr)\wedge\beta$.
The increasing sequences $(\rho_0, \rho_1,\rho_2,\dots)$,  $(\rho'_0, \rho'_1,\rho'_2,\dots)$, and $(\rho''_0, \rho''_1,\rho''_2,\dots)$ are \emph{right-going} in the sense that when the subscript increases by $1$, the subscripted equivalence obtains a new \textquotat{edge} on the right of the earlier edges. By interchanging the role of $\beta$ and $\gamma$, we obtain three increasing \textquotat{left-going} sequences  $(\lambda_0, \lambda_1,\lambda_2,\dots)$,  $(\lambda'_0, \lambda'_1,\lambda'_2,\dots)$, and $(\lambda''_0, \lambda''_1,\lambda''_2,\dots)$. 
Where a right-going sequence \textquotat{reaches} the appropriate left-going one, 
the meet of the two sequences yields an atom of $\Equ A$. Namely,  
for $i\in\set{0,1,\dots, 8}$, $\at(a_i,b_i)=\rho_i\wedge \lambda''_{8-i}\in S$,
$\at(a_ {i+1},b_i)=\rho''_i\wedge \lambda_{8-i}\in S$, and $\at(a_i,a_{i+1})  =\rho'_i\wedge \lambda'_{8-i}\in S$. Furthermore, for $i\in\set{0,\dots,7}$,
$\at(b_{i},b_{i+1})=\bigl(\at(a_{i+1},b_i)\vee \at(a_{i+1},b_{i+1})\bigr)\wedge\alpha\in S$. Hence, for every edge $(x,y)$ of the graph, $\at(x,y)\in S$. 
Therefore, the following lemma implies easily that $X$ generates $\Equ A$. 
\end{proof}

\begin{figure}[ht] 
\centerline{ \includegraphics[width=\figwidthcoeff\textwidth]{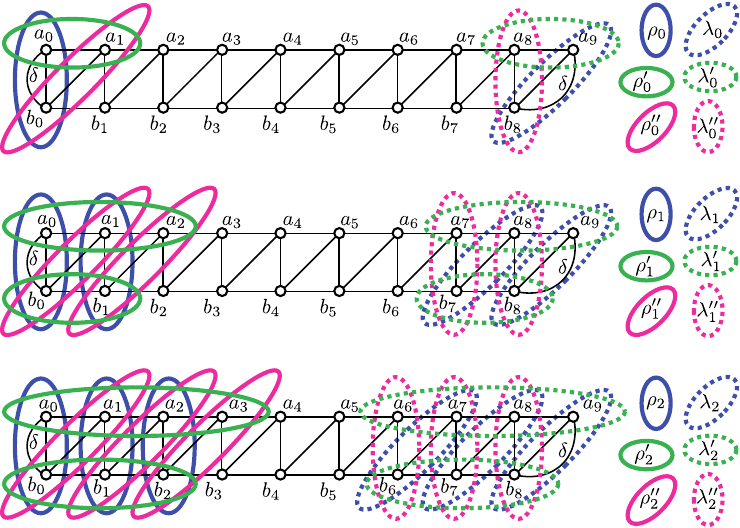}} \caption{Right-going and left-going sequences}\label{figb}
\end{figure}

\begin{lemma}\label{lemma:circle} 
If $3\leq n\in\Nplu=\set{1,2,3,\dots}$,  $A=\set{a_0,a_1,\dots,a_{n-1}}$, and $|A|=n$, then $\{\at(a_{i-1},a_i): i\in\set{1,\dots, n-1}\}\cup\set{\at(a_{n-1},a_0)}$ generates $\Equ A$. 
\end{lemma}

In some form, this easy lemma occurs in several papers; see, e.g., \cite[Lemma 2.2]{czgDEBRauth} and   \cite[Lemma 2.5]{czedlioluoch}.

In 1995, the author visited Ivan Chajda at Palack\'y University in Olomouc.  
The research plan looked easy: by orienting the edges of the graph in Figure \ref{figa} in some way, we should find a small generating set of $\Quo A$. Our first construction was soon developed into a more sophisticated one, and so the first construction does not occur \cite{chajdaczedli}. However, we need the first 
construction\footnote{Its exact details have been lost but the idea of Claim \ref{claim:mc95} is the same.} here 
even though \cite{chajdaczedli} contains a stronger result and we have an even stronger one nowadays.

\begin{figure}[ht] 
\centerline{ \includegraphics[width=\figwidthcoeff\textwidth]{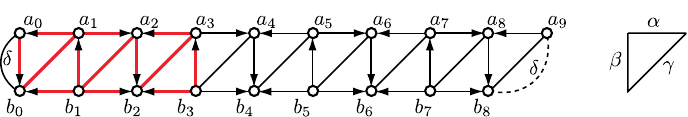}} \caption{Generating a quasiorder lattice}\label{figc}
\end{figure}

 Let $A=\set{a_0,a_1,\dots,a_9, b_0,b_1,\dots,b_8}$ be the 19-element set drawn in Figure \ref{figc}, which is quite similar to Figure \ref{figa}. Some edges are directed by arrowheads, some others are not. 
The figure defines a set $Y_0=\set{\alpha,\beta,\gamma,\delta}$ of quasiorders of $A$ by \eqref{eq:fgdFnwQvl} with the only modification that we cannot walk along a directed edge in the opposite direction. Along an undirected edge, we can walk in both directions. At present, it makes no difference whether an edge is red and thick or not.  For example, $(a_3,a_2), (a_3,a_4)\in\alpha$,  $(a_3,b_2), (b_2,a_3)\in\gamma$, but $(b_4,a_4)\notin\beta$ and $(a_2,a_3),(a_3,a_7)\notin\alpha$.
 For $\epsilon\in Y_0$, denote $\inv\epsilon=\set{(y,x): (x,y)\in\epsilon}\in\Quo A$ the \emph{inverse} of $\epsilon$. Note that $\gamma$ and $\delta$ are equivalences, and so $\igamma=\gamma$ and $\idelta=\delta$. Let  $Y:=Y_0\cup\set{\inv \epsilon: \epsilon \in Y_0}=\set{\alpha, \ialpha, \beta, \ibeta, \gamma,\delta}$.

\begin{claim}\label{claim:mc95} The six-element set $Y$ generates $\Quo A$.
\end{claim}

\begin{proof}[Outline of the proof] For $x,y\in A$, $\qu(x,y)$ denotes the smallest quasiorder containing $(x,y)$. Let $S$ stand for the sublattice generated by $Y$. 
Since $f\colon \Quo A\to \Quo A$ defined by $\mu\mapsto\inv\mu$ is an automorphism of $\Quo A$ and $Y$ is $f$-closed, $S$ is also closed with respect to forming inverses. In particular, whenever $\qu(x,y)$ is in $S$, then so is $\qu(y,x)$; this fact will be used without further explanation. Let us compute; each containment \textquotat{$\in S$} below follows from the earlier ones and $Y\subseteq S$:
\begin{align}
\qu(a_0,b_0) &=\beta\wedge \delta \in S, 
\label{eq:rCpvZnmG2}\\
\qu(a_1,b_0) &= (\alpha\vee \qu(a_0,b_0)) \wedge \gamma\in S ,
\text{ by \eqref{eq:rCpvZnmG2},}       
\label{eq:rCpvZnmG3}\\
\qu(a_1,a_0)&=\alpha\wedge(\qu(a_1,b_0)\vee \qu(b_0,a_0))\in S \text{ by \eqref{eq:rCpvZnmG3} and \eqref{eq:rCpvZnmG2}},   
\label{eq:rCpvZnmG4}\\
\qu(b_1,a_1)&=(\alpha\vee \qu(b_0,a_1))\wedge \beta\in S  
\text{ by \eqref{eq:rCpvZnmG3}},    
\label{eq:rCpvZnmG5}\\
\qu(b_1,b_0)&=\alpha\wedge(\qu(b_1,a_1)\vee \qu(a_1,b_0) )\in S 
\text{ by \eqref{eq:rCpvZnmG5} and \eqref{eq:rCpvZnmG3}},       \label{eq:rCpvZnmG6}\\
\qu(b_1,a_2)&= \gamma \wedge(\qu(b_1,a_1)\vee \alpha )\in S 
\text{ by  \eqref{eq:rCpvZnmG5}},       
\label{eq:rCpvZnmG7}\\
\qu(a_1,a_2)&=\alpha\wedge(\qu(a_1, b_1)\vee \qu(b_1,a_2) )\in S 
\text{ by \eqref{eq:rCpvZnmG5} and \eqref{eq:rCpvZnmG7}},     \label{eq:rCpvZnmG8}\\
\qu(a_2,b_2)&= \beta\wedge(\qu(a_2, b_1)\vee \alpha )\in S 
\text{ by \eqref{eq:rCpvZnmG7}},     
\label{eq:rCpvZnmG9}\\
\qu(b_1,b_2)&= \alpha \wedge(\qu(b_1, a_2)\vee \qu(a_2,b_2))\in S 
\text{ by \eqref{eq:rCpvZnmG7} and \eqref{eq:rCpvZnmG9}},     
\label{eq:rCpvZnmG10}\\
\qu(a_3,b_2)&= \gamma\wedge(\alpha \vee \qu(a_2,b_2) )\in S 
\text{ by \eqref{eq:rCpvZnmG9}},     
\label{eq:rCpvZnmG11}\\
\qu(a_3,a_2)&=  \alpha\wedge(\qu(a_3, b_2)\vee \qu(b_2,a_2) )\in S 
\text{ by \eqref{eq:rCpvZnmG11} and \eqref{eq:rCpvZnmG9}},     \label{eq:rCpvZnmG12}\\
\qu(b_3,a_3)&= \beta \wedge (\alpha\vee \qu(b_2,a_3) )\in S 
\text{ by \eqref{eq:rCpvZnmG11}},     
\label{eq:rCpvZnmG13}\\
\qu(b_3,b_2)&= \alpha\wedge(\qu(b_3, a_3)\vee \qu(a_3,b_2) )\in S 
\text{ by \eqref{eq:rCpvZnmG13} and \eqref{eq:rCpvZnmG11}},     \label{eq:rCpvZnmG14}
\end{align}
and so on. Computations \eqref{eq:rCpvZnmG2}--\eqref{eq:rCpvZnmG14} and the fact that $S$ is closed with respect to forming inverses show that for each thick and red edge $(x,y)$ of the graph, $\qu(x,y)$ and $\qu(y,x)$ are in $S$. The figure and \eqref{eq:rCpvZnmG2}--\eqref{eq:rCpvZnmG14} also show how we can proceed further to the right. Hence,  $\qu(x,y)$ and $\qu(y,x)$ are in $S$ for every edge $(x,y)$ of the graph. Thus, the straightforward counterpart of Lemma  \ref{lemma:circle}  for quasiorder lattices completes the proof of Claim \ref{claim:mc95}.
\end{proof}

In the proof above, $\delta$ was needed only in the first step, \eqref{eq:rCpvZnmG2}. 
This step and the whole proof still work if we omit the dashed curve in Figure \ref{figc} and replace $\delta$ by the equivalence $\at(a_0,b_0)$.  Now  we do not need a left-going sequence of quasiorders. Hence,  and this was a surprise in 1995, we do not need the figure to end on the right. So $A$ can be $\set{a_i:i\in \Nnul}\cup\set{b_i:i\in \Nnul}$, where $\Nnul=\set{0,1,2,\dots}$; this was the moment when an \emph{infinite base set} came into the picture. 

Infinite base sets required new techniques, first for quasiorder lattices, see \cite{chajdaczedli}. The new techniques were soon adapted to infinite equivalence lattices; see, e.g., \cite{czg4genEqLarge}. Later, it appeared that these techniques are useful for finite equivalence lattices; see \cite{czgDEBRauth} and \cite{czedlioluoch}. Due to the results of these two papers, a connection with cryptography has been discovered; see \cite{czgDEBRauth} and, mainly, \cite{czgboolegen}. This connection and many earlier results on four-element generating sets motivate Section \ref{sect:nwthm}, where a new four-element generating set is constructed. To summarize our historical remark: In some sense, most papers mentioned so far and the present one grew from the unpublished proof of Claim \ref{claim:mc95}.

Finally, to conclude this section, note that we can obtain a four-element generating set of $\Quo A$ for $|A|=19$, that is, a stronger result, as follows. (However, this argument  does not show how to step from the class of finite equivalence and quasiorder lattices to that of the infinite ones.) Going after 
\cite{czedlikulin} and using Figure \ref{figa}, add a new $\delta$-curve, a directed one, from $a_1$ to $a_2$. That is, we change $\delta$ to $\delta\vee \qu(a_1,a_2)$. 
By the proof of Claim \ref{claim:Z19}; we obtain all members of $\Equ A$ from $X:=\set{\alpha,\beta,\gamma,\delta}$. Thus, $X$ generates $\Quo A$ by (Kulin's) Lemma \ref{lemma:kulin}.

\section{A new four-element generating set with a special property}\label{sect:nwthm}
The \emph{block count} of an equivalence $\mu\in\Equ A$ is the number $\num\mu$ of blocks of (the partition corresponding to) $\mu$. We say that $X=\set{\mu_1,\mu_2,\mu_3,\mu_4}$ is a 
\emph{four-element generating set of $\Equ A$ with consecutive block counts} if
$X$ generates $\Equ A$ and $\num{\mu_{1+i}}=\num{\mu_1} +i$ for $i\in\set{1,2,3}$. We are going to prove the following theorem.

\begin{theorem}\label{thm:main} 
If the number of elements of a finite set $A$ is six or it is at least eight, then  $\Equ A$ has a four-element generating set with consecutive block counts.
\end{theorem}

Similar properties (namely, \textquotat{same block counts} and \textquotat{the difference between the block counts   $\leq 2$}) have been studied in \cite{czg4ghoriz} and \cite{czg4gw2}; the property we consider in this section is more difficult to fulfill. 
Despite some similarities with  \cite{czg4ghoriz} and \cite{czg4gw2} in the approach, the present paper remains self-contained.

\begin{remark} We know that if $|A|<6$, then $\Equ A$ has no four-element generating set with consecutive block counts; we guess the same for  $|A|=7$. 
\end{remark}

A pair $(\mu,\nu)$ of elements of $\Equ A$ is \emph{complementary} if $\mu\vee\nu=\latop A$, the top element of $\Equ A$, and $\mu\wedge\nu=\labot A$, the bottom element of $\Equ A$.

\begin{definition}[\cite{czg4ghoriz}]\label{def:elig}
A $7$-tuple  $\asys=(A;\alpha,\beta,\gamma,\delta; u,v)$ is called an \emph{eligible system} if $A$ is a nonempty set, $\set{\alpha,\beta,\gamma,\delta}$ is a generating set of $\Equ A$, 
and the pairs $(\alpha,\delta)$, $\bigl(\beta,\gamma\vee \at(u,v)\bigr)$, and $\bigl(\beta\vee\at(u,v),\gamma\bigr)$ are complementary.
\end{definition}

 For $\rho\subseteq (A')^2$, $\peq(\rho)$ will denote the smallest equivalence of $A'$ that includes  $\rho$. For distinct elements $x,y\in A'$, let $\pat(x,y):=\peq(\set{(x,y)})$. 
 The lattice operations in $\Equ{A'}$ will be denoted by $\pvee$ and $\pwedge$.

\begin{lemma}\label{lemma:tstp}
Let $\asys$ be an eligible system with components denoted as in Definition \ref{def:elig}. Assume that $u',v'\notin A$ and $u'\neq v'$. Let $A':=A\cup\set{u',v'}$, $\alpha':=\peq(\alpha) \pvee \pat(u,u')$, $\beta':=\peq(\beta)\pvee \pat(u,v')$, $\gamma':=\peq(\gamma)\pvee\pat(v,v')$, $\delta':=\peq(\delta)\pvee\pat(u',v')$. Then $\asys':=(A';\alpha',\beta',\gamma',\delta';u',v')$ is an eligible system, too.
 Furthermore, if $\Phi:=\set{\alpha,\beta,\gamma,\delta}$ is of consecutive block counts, then so is $\Phi':=\set{\alpha',\beta',\gamma',\delta'}$.
\end{lemma}

\begin{proof}[Proof]  The situation is visualized in Figure \ref{fig1}, where the blocks of some elements, all important elements from our perspective,  are drawn. The three blocks drawn by solid lines are blocks of some members of $\Phi\subseteq\Equ A$. The seven blocks drawn in non-solid line styles (dotted and various kinds of dashed) are blocks of the equivalences belonging to $\Phi'\subseteq \Equ{A'}$. The figure uses different line styles or distinct colors for the blocks of different equivalences, but we use the same color for $\epsilon\in\Phi$ and $\epsilon'$. Note that the geometrically large blocks on the left could be singletons and, on the other hand, $u/\alpha:=\set{x: (x,u)\in\alpha}$ and $v/\gamma$ can be but need not be disjoint.  Not all blocks of all $\epsilon$ and $\epsilon'$ are drawn for $\epsilon\in\Phi$. However, for any $x\in A$ and $\epsilon\in \Phi$, if the block $x/\epsilon$ is not drawn, then  $x/\epsilon=x/\epsilon'$. The last sentence of Lemma \ref{lemma:tstp} follows from the trivial fact that $\num{\epsilon'}=1+\num\epsilon$ holds for every $\epsilon\in\Phi$. 
Applying a lemma from \cite{czg4ghoriz} twice (in a \textquotat{twisted way} and in a \textquotat{straight way}), we could derive the rest of  Lemma \ref{lemma:tstp} from  \cite{czg4ghoriz}. To keep the paper self-contained, we give a different and direct proof.

\begin{figure}[ht] 
\centerline{ \includegraphics[width=\figwidthcoeff\textwidth]{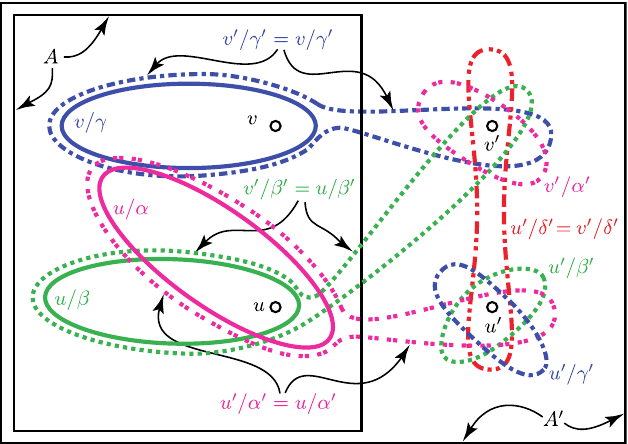}} \caption{Illustrating the proof of Lemma \ref{lemma:tstp}}\label{fig1}
\end{figure}

The existence of an $x\in u/\beta \wedge v/\gamma$ would violate the conjunction of $\beta\wedge(\gamma\vee\at(u,v))=\labot A$ and $\gamma\wedge(\beta\vee\at(u,v))=\labot A$ ---call them the \emph{meet conditions} for $\beta$ and $\gamma$--- and $u\neq v$. Thus,  $u/\beta$ and $v/\gamma$ are disjoint. 

By the two paragraphs above, Figure \ref{fig1} faithfully represents the situation and contains all the details the proof needs.  Hence, it is straightforward to verify that the three pairs in 
Definition \ref{def:elig} for $\Equ{A'}$  are complementary.
Let $S$ and $E$ denote the sublattice generated by $\Phi'$ in $\Equ{A'}$ and the sublattice $\{\mu\in \Equ{A'}$ : both $u'/\mu$ and $v'/\mu$ are singletons$\}$. Then  $f\colon \Equ A\to E$ defined by  $\mu\mapsto \peq(\mu)$ is a lattice isomorphism. 

Observe that  $\set{u'}$ is a singleton block of  
$\beta'\pvee \gamma'$. Furthermore, $\set{v'}$ is a singleton block of both $\alpha'$ and $\delta'\pwedge (\beta'\pvee \gamma')$. Thus $\set{v'}$ is a singleton block of
$\alpha'\pvee \bigl(\delta'\pwedge (\beta'\pvee \gamma')\bigr)$. Therefore, with
\begin{equation*}
\kappa:=(\beta'\pvee \gamma')  \pwedge \Bigl(\alpha'\pvee \bigl(\delta'\pwedge (\beta'\pvee \gamma')\bigr)\Bigr),
\end{equation*}
$|u'/\kappa|=|v'/\kappa|=1$. Using the fact that $\epsilon\subseteq \epsilon'$ for all $\epsilon\in X$, the join condition $\beta\vee\at(u,v)\vee\gamma=\latop A$  for $\beta$ and $\gamma$, and $(u,v)\in\beta'\pvee\gamma'$, we obtain that $A^2\subseteq \beta'\pvee \gamma'$.
By the previous two \textquotat{$\subseteq$} inclusions, $\delta\subseteq \delta'\pwedge (\beta'\pvee \gamma')$. Using this fact, $\alpha\subseteq \alpha'$, and the join condition for the complementary pair  $(\alpha,\delta)$, we obtain that $A^2\subseteq \kappa$.
Combining this with $|u'/\kappa|=|v'/\kappa|=1$, we have that $f(\latop A)=\kappa\in S$. Thus, for all $\epsilon \in \Phi$, $f(\epsilon)=f(\latop A)\wedge \epsilon'\in S$, whereby $f(\Phi)\subseteq S$. Since $\Phi$ generates $\Equ A$ and $f\colon \Equ A\to E$ is an isomorphism, we obtain that $E \subseteq S$.  In particular, $\pat(u,v)=f(\at(u,v))\in S$. As the following equalities are clear by the figure, we obtain further elements of $S$ as follows:
\begin{align}
\pat(v,v')&=(\pat(u,v)\pvee \beta')\pwedge \gamma'\in S, \label{eq:zwnLbrBrta}\\
\pat(v',u)&=(\pat(u,v)\vee \pat(v,v'))\pwedge \beta'\in S, \nonumber \\
\pat(v',u')&= (\pat(v',u)\pvee \alpha')\pwedge \delta'  \in S \text{, and} \label{eq:zwnLbrBrtb}\\
\pat(u',u)&=\alpha' \pwedge ( \pat(u',v')\pvee \pat(v',u) ) \in S.\label{eq:zwnLbrBrtc}
\end{align}
Finally, since $E\subseteq S$ and we have \eqref{eq:zwnLbrBrta}, \eqref{eq:zwnLbrBrtb}, and \eqref{eq:zwnLbrBrtc}, Lemma \ref{lemma:circle}  implies that $S=\Equ{A'}$. This completes the proof of Lemma \ref{lemma:tstp}.
\end{proof}

\begin{lemma}\label{lemma6elig} With $A=\set{1,2,\dots,6}$, 
\begin{align}
\alpha&:=\eq(12;3;45;6), \label{eq6dWk1}\\
\beta&:=\eq(1;2;34;5;6) ,\label{eq6dWk2}\\
\gamma&:=\eq(13;24;56) ,\text{ and} \label{eq6dWk3}\\
\delta&:=\eq(146;235), \label{eq6dWk4}
\end{align}
$\asys=(A;\alpha,\beta,\gamma,\delta; 4,6)$ is an eligible system with consecutive block counts.
\end{lemma}

\begin{proof} 
Let $\Phi:=\set{\alpha,\beta,\gamma,\delta}$, and let $S$ stand for the sublattice generated by $S$. The labels above the equality signs will indicate which members of $S$ imply that the equivalences on the left of these equality signs belong to $S$.
\allowdisplaybreaks{
\begin{align}
\eq(12;345;6)  & \deqref{eq6dWk1}{eq6dWk2} \eq(12;3;45;6)\vee \eq(1;2;34;5;6) ,\label{eq6dWk5}\\
\eq(1234;56)  & \deqref{eq6dWk2}{eq6dWk3}     \eq(1;2;34;5;6)\vee \eq(13;24;56) ,\label{eq6dWk8}\\
\eq(12;3;4;5;6) &  \deqref{eq6dWk1}{eq6dWk8} \eq(12;3;45;6)\wedge \eq(1234;56) ,\label{eq6dWk9}\\
\eq(1;2;35;4;6)  & \deqref{eq6dWk4}{eq6dWk5} \eq(146;235)\wedge \eq(12;345;6) ,\label{eq6dWk10}\\
\eq(14;23;5;6)  & \deqref{eq6dWk4}{eq6dWk8}  \eq(146;235)\wedge \eq(1234;56) ,\label{eq6dWk11}\\
\eq(1;2;345;6)  & \deqref{eq6dWk2}{eq6dWk10} \eq(1;2;34;5;6)\vee \eq(1;2;35;4;6) ,\label{eq6dWk14}\\
\eq(1356;24)   & \deqref{eq6dWk3}{eq6dWk10} \eq(13;24;56)\vee \eq(1;2;35;4;6) ,\label{eq6dWk16}\\
\eq(14;235;6)  &  \deqref{eq6dWk10}{eq6dWk11} \eq(1;2;35;4;6)\vee \eq(14;23;5;6) ,\label{eq6dWk18}\\
\eq(1;2;3;45;6) &  \deqref{eq6dWk1}{eq6dWk14} \eq(12;3;45;6)\wedge \eq(1;2;345;6) ,\label{eq6dWk19}\\
\eq(16;2;35;4)  & \deqref{eq6dWk4}{eq6dWk16} \eq(146;235)\wedge \eq(1356;24) ,\label{eq6dWk21}\\
\eq(13;2456)  & \deqref{eq6dWk3}{eq6dWk19}  \eq(13;24;56)\vee \eq(1;2;3;45;6) ,\label{eq6dWk25}\\
\eq(126;35;4)  &  \deqref{eq6dWk9}{eq6dWk21} \eq(12;3;4;5;6)\vee \eq(16;2;35;4) ,\label{eq6dWk26}\\
\eq(145;23;6)  & \deqref{eq6dWk11}{eq6dWk19}  \eq(14;23;5;6)\vee \eq(1;2;3;45;6) ,\label{eq6dWk27}\\
\eq(15;2;3;4;6)  & \deqref{eq6dWk16}{eq6dWk27}  \eq(1356;24)\wedge \eq(145;23;6) ,\label{eq6dWk31}\\
\eq(1;26;3;4;5) &  \deqref{eq6dWk25}{eq6dWk26} \eq(13;2456)\wedge \eq(126;35;4) ,\label{eq6dWk33}\\
\eq(135;2;4;6)  &  \deqref{eq6dWk10}{eq6dWk31} \eq(1;2;35;4;6)\vee \eq(15;2;3;4;6) ,\label{eq6dWk47}\\
\eq(14;2356)  & \deqref{eq6dWk18}{eq6dWk33}  \eq(14;235;6)\vee \eq(1;26;3;4;5) ,\label{eq6dWk60}\\
\eq(13;2;4;5;6) &   \deqref{eq6dWk3}{eq6dWk47} \eq(13;24;56)\wedge \eq(135;2;4;6) ,\label{eq6dWk79}\\
\eq(1;2;3;4;56)  &   \deqref{eq6dWk3}{eq6dWk60} \eq(13;24;56)\wedge \eq(14;2356). \label{eq6dWk80}
\end{align}
}
In particular, 
$\at(1,2)\in S$ by  \eqref{eq6dWk9},        
$\at(2,6)\in S$ by \eqref{eq6dWk33}, 
$\at(6,5)\in S$ by \eqref{eq6dWk80}, 
$\at(5,4)\in S$ by \eqref{eq6dWk19}, 
$\at(4,3)\in S$ by \eqref{eq6dWk2}, and 
$\at(3,1)\in S$ by \eqref{eq6dWk79}. Hence, $\Phi$ is a generating set by  Lemma \ref{lemma:circle}. Clearly,
$\Phi$ is of consecutive block counts. 
It is easy to check that the pairs in Definition \ref{def:elig} are complementary. 
Thus, $\asys$ is an eligible system, proving Lemma \ref{lemma6elig}.
\end{proof}
The author has created a program package called \textquotat{equ2024p}, available from his website \href{http://tinyurl.com/g-czedli/}{http://tinyurl.com/g-czedli/}. This program package can also \textquotat{prove} that $\Phi$ generates $\Equ A$, but verifying the programs is much more difficult than verifying the proofs of Lemmas \ref{lemma6elig} and (the next) \ref{lemma9elig}.

\begin{lemma}\label{lemma9elig} With $A=\set{1,2,\dots,9}$, 
\begin{align}
\alpha&:=\eq(158;2;3;47;69), \label{eq9j1}\\
\beta&:=\eq(1;23;4;56;78;9), \label{eq9j2}\\
\gamma&:=\eq(135;268;4;79), \text{ and} \label{eq9j3}\\
\delta&:=\eq(16;257;3489) ,\label{eq9j4}
\end{align}
$\asys=(A;\alpha,\beta,\gamma,\delta; 1,4)$ is an eligible system with consecutive block counts.
\end{lemma}

The proof of this lemma is similar to but more than three times longer than the previous proof. As the reader would hardly 
enjoy such an amount of technicalities, the proof goes into the Appendix  (Section \ref{sect:append}) of the paper.

Now, we are in the position to prove our theorem.

\begin{proof}[Proof of Theorem \ref{thm:main}] Combine Lemmas \ref{lemma:tstp}, \ref{lemma6elig}, and \ref{lemma9elig}.
\end{proof}

\section{Appendix}\label{sect:append}

\begin{proof}[Proof of Lemma \ref{lemma9elig}] 
The idea is the same as in the proof of Lemma \ref{lemma6elig} but the concrete computations are different.
Again, $S$ denotes the sublattice generated by $\set{\alpha,\dots,\delta}$. Let us compute; 
to avoid line breaks, the formatting will be slightly different from that of \eqref{eq6dWk5}--\eqref{eq6dWk80}.
\allowdisplaybreaks{
\begin{align}
&\eq(1456789;23) \shdeqref{eq9j1}{eq9j2} \eq(158;2;3;47;69)\vee \eq(1;23;4;56;78;9), \label{eq9j6}  \\
& \eq(15;2;3;4;6;7;8;9) \shdeqref{eq9j1}{eq9j3} \eq(158;2;3;47;69)\wedge \eq(135;268;4;79), \label{eq9j7}  \\
&\eq(12356789;4) \shdeqref{eq9j2}{eq9j3} \eq(1;23;4;56;78;9)\vee \eq(135;268;4;79), \label{eq9j9}  \\
&\eq(158;2;3;4;69;7) \shdeqref{eq9j1}{eq9j9} \eq(158;2;3;47;69)\wedge \eq(12356789;4), \label{eq9j10}  \\
&\eq(156;23;4;78;9) \shdeqref{eq9j2}{eq9j7} \eq(1;23;4;56;78;9)\vee \eq(15;2;3;4;6;7;8;9), \label{eq9j11}  \\
&\eq(15;2;3;4;68;79) \shdeqref{eq9j3}{eq9j6} \eq(135;268;4;79)\wedge \eq(1456789;23), \label{eq9j12}  \\
&\eq(16;2;3;489;57) \shdeqref{eq9j4}{eq9j6} \eq(16;257;3489)\wedge \eq(1456789;23), \label{eq9j13}  \\
&\eq(12567;3489) \shdeqref{eq9j4}{eq9j7} \eq(16;257;3489)\vee \eq(15;2;3;4;6;7;8;9), \label{eq9j14}  \\
&\eq(16;257;389;4) \shdeqref{eq9j4}{eq9j9} \eq(16;257;3489)\wedge \eq(12356789;4), \label{eq9j15}  \\
&\eq(1456789;2;3) \shdeqref{eq9j1}{eq9j12} \eq(158;2;3;47;69)\vee \eq(15;2;3;4;68;79), \label{eq9j17}  \\
& \eq(1;2;3;4;56;7;8;9) \shdeqref{eq9j2}{eq9j14} \eq(1;23;4;56;78;9)\wedge \eq(12567;3489), \label{eq9j18}  \\
&\eq(15;26;3;4;7;8;9) \shdeqref{eq9j3}{eq9j14} \eq(135;268;4;79)\wedge \eq(12567;3489), \label{eq9j19}  \\
& \eq(16;2;3;4;5;7;8;9) \shdeqref{eq9j4}{eq9j11} \eq(16;257;3489)\wedge \eq(156;23;4;78;9), \label{eq9j20}  \\
&\eq(15689;2;3;47) \shdeqref{eq9j1}{eq9j18} \eq(158;2;3;47;69)\vee \eq(1;2;3;4;56;7;8;9), \label{eq9j27}  \\
&\eq(1;2;3;4;56;78;9) \shdeqref{eq9j2}{eq9j17} \eq(1;23;4;56;78;9)\wedge \eq(1456789;2;3), \label{eq9j29}  \\
&\eq(12356;4;78;9) \shdeqref{eq9j2}{eq9j19} \eq(1;23;4;56;78;9)\vee \eq(15;26;3;4;7;8;9), \label{eq9j30}  \\
&\eq(123568;4;79) \shdeqref{eq9j3}{eq9j18} \eq(135;268;4;79)\vee \eq(1;2;3;4;56;7;8;9), \label{eq9j31}  \\
&\eq(158;2;3;4;6;7;9) \shdeqref{eq9j1}{eq9j31} \eq(158;2;3;47;69)\wedge \eq(123568;4;79), \label{eq9j42}  \\
&\eq(135;26;4;7;8;9) \shdeqref{eq9j3}{eq9j30} \eq(135;268;4;79)\wedge \eq(12356;4;78;9), \label{eq9j46}  \\
&\eq(16;2;3;4;5;7;89) \shdeqref{eq9j4}{eq9j27} \eq(16;257;3489)\wedge \eq(15689;2;3;47), \label{eq9j47}  \\
&\eq(16;25;38;4;7;9) \shdeqref{eq9j4}{eq9j31} \eq(16;257;3489)\wedge \eq(123568;4;79), \label{eq9j49}  \\
&\eq(1256;3;4;78;9) \shdeqref{eq9j19}{eq9j29} \eq(15;26;3;4;7;8;9)\vee \eq(1;2;3;4;56;78;9), \label{eq9j58}  \\
&\eq(1358;269;47) \shdeqref{eq9j1}{eq9j46} \eq(158;2;3;47;69)\vee \eq(135;26;4;7;8;9), \label{eq9j63}  \\
&\eq(15678;23;4;9) \shdeqref{eq9j2}{eq9j42} \eq(1;23;4;56;78;9)\vee \eq(158;2;3;4;6;7;9), \label{eq9j65}  \\
&\eq(156;23;4;789) \shdeqref{eq9j2}{eq9j47} \eq(1;23;4;56;78;9)\vee \eq(16;2;3;4;5;7;89), \label{eq9j66}  \\
&\eq(123567;489) \shdeqref{eq9j13}{eq9j46} \eq(16;2;3;489;57)\vee \eq(135;26;4;7;8;9), \label{eq9j74}  \\
&\eq(1256;378;4;9) \shdeqref{eq9j29}{eq9j49} \eq(1;2;3;4;56;78;9)\vee \eq(16;25;38;4;7;9), \label{eq9j81}  \\
&\eq(12356;4;789) \shdeqref{eq9j30}{eq9j47} \eq(12356;4;78;9)\vee \eq(16;2;3;4;5;7;89), \label{eq9j83}  \\
&\eq(1256;3;4;789) \shdeqref{eq9j47}{eq9j58} \eq(16;2;3;4;5;7;89)\vee \eq(1256;3;4;78;9), \label{eq9j92}  \\
&\eq(15;2;3;4;6;79;8) \shdeqref{eq9j3}{eq9j66} \eq(135;268;4;79)\wedge \eq(156;23;4;789), \label{eq9j93}  \\
&\eq(15;26;3;4;79;8) \shdeqref{eq9j3}{eq9j92} \eq(135;268;4;79)\wedge \eq(1256;3;4;789), \label{eq9j95}  \\
& \eq(1;2;38;4;5;6;7;9) \shdeqref{eq9j4}{eq9j63} \eq(16;257;3489)\wedge \eq(1358;269;47), \label{eq9j96}  \\
&\eq(16;2;3;4;57;8;9) \shdeqref{eq9j4}{eq9j65} \eq(16;257;3489)\wedge \eq(15678;23;4;9), \label{eq9j97}  \\
&\eq(15;26;38;4;7;9) \shdeqref{eq9j14}{eq9j63} \eq(12567;3489)\wedge \eq(1358;269;47), \label{eq9j100}  \\
&\eq(1256;37;4;8;9) \shdeqref{eq9j74}{eq9j81} \eq(123567;489)\wedge \eq(1256;378;4;9), \label{eq9j111}  \\
&\eq(158;2;3;4679) \shdeqref{eq9j1}{eq9j93} \eq(158;2;3;47;69)\vee \eq(15;2;3;4;6;79;8), \label{eq9j112}  \\
&\eq(125689;347) \shdeqref{eq9j1}{eq9j111} \eq(158;2;3;47;69)\vee \eq(1256;37;4;8;9), \label{eq9j116}  \\
&\eq(158;2679;3;4) \shdeqref{eq9j10}{eq9j95} \eq(158;2;3;4;69;7)\vee \eq(15;26;3;4;79;8), \label{eq9j121}  \\
&\eq(15;2;368;4;79) \shdeqref{eq9j12}{eq9j96} \eq(15;2;3;4;68;79)\vee \eq(1;2;38;4;5;6;7;9), \label{eq9j125}  \\
&\eq(15;2368;4;79) \shdeqref{eq9j12}{eq9j100} \eq(15;2;3;4;68;79)\vee \eq(15;26;38;4;7;9), \label{eq9j126}  \\
&\eq(12568;379;4) \shdeqref{eq9j12}{eq9j111} \eq(15;2;3;4;68;79)\vee \eq(1256;37;4;8;9), \label{eq9j127}  \\
&\eq(15679;2;3;4;8) \shdeqref{eq9j93}{eq9j97} \eq(15;2;3;4;6;79;8)\vee \eq(16;2;3;4;57;8;9), \label{eq9j163}  \\
&\eq(15;2;3;4;69;7;8) \shdeqref{eq9j1}{eq9j163} \eq(158;2;3;47;69)\wedge \eq(15679;2;3;4;8), \label{eq9j176}  \\
& \eq(1;23;4;5;6;7;8;9) \shdeqref{eq9j2}{eq9j126} \eq(1;23;4;56;78;9)\wedge \eq(15;2368;4;79), \label{eq9j177}  \\
& \eq(1;2;3;49;5;6;7;8) \shdeqref{eq9j4}{eq9j112} \eq(16;257;3489)\wedge \eq(158;2;3;4679), \label{eq9j184}  \\
&\eq(16;25;34;7;89) \shdeqref{eq9j4}{eq9j116} \eq(16;257;3489)\wedge \eq(125689;347), \label{eq9j187}  \\
& \eq(1;27;3;4;5;6;8;9) \shdeqref{eq9j4}{eq9j121} \eq(16;257;3489)\wedge \eq(158;2679;3;4), \label{eq9j189}  \\
&\eq(16;25;39;4;7;8) \shdeqref{eq9j4}{eq9j127} \eq(16;257;3489)\wedge \eq(12568;379;4), \label{eq9j191}  \\
&\eq(15;2;36;4;79;8) \shdeqref{eq9j83}{eq9j125} \eq(12356;4;789)\wedge \eq(15;2;368;4;79), \label{eq9j257}  \\
&\eq(158;2;34679) \shdeqref{eq9j1}{eq9j257} \eq(158;2;3;47;69)\vee \eq(15;2;36;4;79;8), \label{eq9j327}  \\
&\eq(135;26789;4) \shdeqref{eq9j3}{eq9j176} \eq(135;268;4;79)\vee \eq(15;2;3;4;69;7;8), \label{eq9j344}  \\
&\eq(16;235789;4) \shdeqref{eq9j15}{eq9j177} \eq(16;257;389;4)\vee \eq(1;23;4;5;6;7;8;9), \label{eq9j402}  \\
&\eq(16;2359;4;7;8) \shdeqref{eq9j177}{eq9j191} \eq(1;23;4;5;6;7;8;9)\vee \eq(16;25;39;4;7;8), \label{eq9j716}  \\
& \eq(1;2;3;4;58;6;7;9) \shdeqref{eq9j1}{eq9j402} \eq(158;2;3;47;69)\wedge \eq(16;235789;4), \label{eq9j947}  \\
& \eq(1;2;3;4;5;6;78;9) \shdeqref{eq9j2}{eq9j344} \eq(1;23;4;56;78;9)\wedge \eq(135;26789;4), \label{eq9j954}  \\
&\eq(1;29;35;4;6;7;8) \shdeqref{eq9j63}{eq9j716} \eq(1358;269;47)\wedge \eq(16;2359;4;7;8), \label{eq9j1269}  \\
& \eq(1;2;34;5;6;7;8;9) \shdeqref{eq9j187}{eq9j327} \eq(16;25;34;7;89)\wedge \eq(158;2;34679), \label{eq9j1655}  \\
&\eq(1;23569;4;78) \shdeqref{eq9j2}{eq9j1269} \eq(1;23;4;56;78;9)\vee \eq(1;29;35;4;6;7;8), \label{eq9j2485}  \\
& \eq(1;2;3;4;5;69;7;8) \shdeqref{eq9j1}{eq9j2485} \eq(158;2;3;47;69)\wedge \eq(1;23569;4;78). \label{eq9j5501} 
\end{align}
}
In particular, 
$\at(5,1)\in S$ by    \eqref{eq9j7},
$\at(1,6)\in S$ by    \eqref{eq9j20},
$\at(6.9)\in S$ by    \eqref{eq9j5501},
$\at(9,4)\in S$ by    \eqref{eq9j184},
$\at(4,3)\in S$ by    \eqref{eq9j1655},
$\at(3,2)\in S$ by    \eqref{eq9j177},
$\at(2,7)\in S$ by    \eqref{eq9j189},
$\at(7,8)\in S$ by    \eqref{eq9j954}, and
$\at(8,5)\in S$ by    \eqref{eq9j947}. Hence, Lemma \ref{lemma:circle} implies that $S=\Equ A$, as
required. The rest of the proof is trivial; hence, we omit it.
\end{proof}

\end{document}